\documentclass[12pt, reqno]{amsart}
\usepackage{amsmath, amsthm, amscd, amsfonts, amssymb, graphicx, color}
\usepackage[bookmarksnumbered, colorlinks, plainpages]{hyperref}
\hypersetup{colorlinks=true,linkcolor=red, anchorcolor=green, citecolor=cyan, urlcolor=red, filecolor=magenta, pdftoolbar=true}

\newtheorem{theorem}{Theorem}[section]
\newtheorem{lemma}[theorem]{Lemma}

\newtheorem{corollary}[theorem]{Corollary}
\theoremstyle{definition}

\newtheorem{example}[theorem]{Example}

\theoremstyle{remark}
\newtheorem{remark}[theorem]{Remark}
\numberwithin{equation}{section}

\newcommand{\be}{\begin{equation}}
\newcommand{\ee}{\end{equation}}

\newcommand{\NN}{\mathbb{N}}

\begin{document}
\setcounter{page}{1}

\title[Inequalities on the generalized and the joint spectral radius of Hadamard products]{Inequalities on the generalized and the joint spectral radius of bounded sets of positive operators on Banach function and sequence spaces}

\author[K. Bogdanovi\'c]{Katarina Bogdanovi\'c$^*$}

\address{$^*$ University of Belgrade, Faculty of Mathematics, Studentski trg 16, P. O. box 550, 11 000 Belgrade, Serbia}
\email{\textcolor[rgb]{0.00,0.00,0.84}{katarinabgd77@gmail.com}}





\subjclass[2010]{Primary 47B65; Secondary 15A42, 15A60.}

\keywords{Hadamard-Schur product; Spectral radius; Non-negative matrices; Positive operators; Sequence spaces.}

\date{Received: xxxxxx; Revised: yyyyyy; Accepted: zzzzzz.
\newline \indent $^{*}$ Corresponding author}

\begin{abstract}
In this article we prove new inequalities for the generalized and the joint spectral radius of bounded sets of positive operators on Banach function and sequence spaces, in particular some inequalities for positive kernel operators on $L^2(X, \mu)$ that generalize, extend and refine several of results obtained recently.
\end{abstract}\maketitle

\section{Introduction and Preliminaries}
\vspace{1mm}

Let $(X, \mathfrak M, \mu)$ be a non -empty space $X$ with a $\sigma$-finite positive measure $\mu$ on a $\sigma$-algebra $\mathfrak M$, i. e. $X$ is the union of at most countably numbers of sets of finite measure.We denote by $M(\mu)$ the collection of all measurable complex valued functions such that functions being equal almost everywhere will be identified. 
 A Banach space $L\subseteq M(\mu)$ is called a {\it Banach function space} on $(X, \mathfrak M, \mu)$ if for any $f\in L$ and $g\in M(\mu)$ the following implication holds $|g|\le|f|\Rightarrow g\in L$ and $\|g\| \le \|f\|$. We assume that $X$ is the carrier of $L$. Carrier is defined to be the complement in $X$ of a maximal set $Y\in\mathfrak M$ such that $f \chi_{Y}=0$ a.e., for all $f\in L,$ where $\chi_Y$ stands for the characteristic function of the set $Y\in\mathfrak M$ (see \cite{Za83} for more details). In the special case of $(X, \mathfrak M, \mu)$, where $X=\NN$, $\mathfrak M$ is a $\sigma$-algebra of all subsets of $\NN$ and $\mu$ is a counting measure on $\NN$, the Banach function space will be called {\it Banach sequence space}. Clearly the possibility that $X=\{1, \ldots , N\}$ for some finite $N \in \NN$ is also allowed.   

In the sequel we will refer to a family $\mathcal{L}$ as the collection of all Banach sequence spaces
$L$ satisfying the condition that $e_n = \chi_{\{n\}} \in L$ and
$\|e_n\|_L=1$ for all $n \in \NN$.

Examples of Banach function spaces are the classical Lebesgue spaces $L^p(\mu)$  for $1\le p \le\infty$, Orlicz spaces (\cite{KM99}), Lorentz and Marcinkiewicz spaces (\cite{BS88}, \cite{KPS82}) while examples of Banach sequence spaces are well known $l^p$ spaces for $1\le p \le\infty$.

The norm on a Banach function space $L$ is called {\it order continuous norm} if $\|f_n\|_L\to 0$ as $n \to\infty$ for every positive, decreasing sequence $\{f_n\}_{n\in\mathbb{N}}\subset M(\mu)$ such that $\inf\{f_n \in M(\mu): n\in \NN\}=0$, which is denoted by $f_n \downarrow 0$. For $1\le p<\infty $ $L^p(\mu)$ spaces are classical examples of spaces with an order continuous norm. 
Moreover, every reflexive Banach function space has an order continuous norm. We will be particularly interested in Banach function spaces $L$ such that $L$ and its dual space $L^*$ have order continuous norm. Specially well known among them are spaces $L^p(\mu)$, $1< p<\infty $. An example of a such non-reflexive Banach function space is $L=c_0\in\mathcal{L}$, where $L^*=(c_0)^*=l^1\in\mathcal{L}$.

By an {\it operator} on a Banach function space $L$ we always mean a linear operator on $L$. An operator $A$ on $L$ is said to be {\it positive} if it maps nonnegative functions to nonnegative ones, that is, $AL_+\subset L_+$, where $L_+$ denotes the positive cone $L_+=\{f\in L : f\ge 0 \; \mathrm{a. e}\}$. Given operators $A$ and $B$ on $L$, we write $A\ge B$ if the operator $A-B$ is positive. Recall that a positive operator on $L$ is always bounded , i. e., its operator norm 
\be
\|A\|=\sup\{\|Ax\|_L : x\in L, \|x\|_L \le 1\}=\sup\{\|Ax\|_L : x\in L_+, \|x\|_L \le 1\}
\label{einfuehrung}
\ee
is finite (the second equality in \eqref{einfuehrung} follows from $|Af|\le A|f|$ for $f\in L$). (see e. g. \cite[Theorem 1.7]{P09}). 

An operator $A$ on a Banach function space $L$ is called a {\it kernel operator} if there exists $\mu\times\mu$-measurable function $a(x, y)$ on $X\times X$ such that, for all $f\in L$ and for almost all $x \in X$
$$\int_X |a(x, y)f(y)|\, d\mu(y)< \infty \ \ \ {\rm and} \ \
  (Af)(x)=\int_X a(x, y)f(y)\, d\mu(y).$$
Recall that a kernel operator $A$ is positive iff its kernel $a$ is non-negative function almost everywhere.

Let $L$ be a Banach function space such that $L$ and $L^*$ have order continuous norms and let $A$ and $B$ be positive kernel operators on $L$. Then the Hausdorff measure of non-compactness of $A$, denoted by $\gamma (A)$, is defined as follows 
$$\gamma (A)=\inf\left\{\delta >0 : \;\; \mathrm{there}\;\; \mathrm{is}\;\; \mathrm{a}\;\; \mathrm{finite}\;\; M \subset L \;\; \mathrm{such} \;\; \mathrm{that}\;\; A(D_L)\subset M+ \delta D_L \right\},$$
where $D_L=\{f \in L : \|f\|_L \le 1\}$. The basic properties of the Hausdorff measure of non-compactness are $\gamma (A)\le \|A\|$, $\gamma (A+B)\le \gamma (A)+\gamma (B)$, $\gamma (AB)
\le \gamma (A)\gamma (B)$, and $\gamma (\alpha A)=\alpha \gamma (A)$ for $\alpha \ge 0$. Moreover, $0\le A \le B$ implies $\gamma (A)\le \gamma (B)$ (see e.g. \cite[Corollary 4.3.7 and Corollary 3.7.3]{M91}).

 The essential spectral radius of $A$, denoted by $\rho_{ess} (A)$ is defined as follows
 \be
 \nonumber
 \rho_{ess} (A)=max\{|\lambda|: \lambda\in\sigma_{ess} (A)\}
 \ee
 where $\sigma_{ess} (A)$ is the essential spectral radius of $A$ (see \cite{M91} for more details). Then
\be
\rho_{ess} (A)=\lim_{j \to \infty} \gamma (A^j)^{1/j}=\inf_{j \in \NN} \gamma (A^j)^{1/j}
\label{uvod}
\ee
 Recall that if $L=L^2(X, \mu)$ 
  $\gamma (A^*)=\gamma (A)$ and $\rho_{ess} (A^*)=\rho_{ess} (A)$, where $A^*$ denotes the adjoint of $A$. Note that the equalities (\ref{uvod}) and
  $\rho_{ess} (A^*)=\rho_{ess} (A)$  are valid for any bounded operator $A$ on a given complex Banach space $L$ (see e.g. \cite[Theorem 4.3.13 and Proposition 4.3.11]{M91}).


Let $A$ and $B$ be positive kernel operators on a Banach function space $L$ with kernels $a$ and $b$ respectively, and $\alpha \ge 0$.
The \textit{Hadamard or (Schur) product} $A \circ B$ of $A$ and $B$ is the kernel operator with kernel equal to $a(x, y)b(x, y)$ at point $(x, y)\in X\times X$, i. e. $A \circ B$ 
 is the positive kernel operator given by
\be
\nonumber
(A \circ B)f(x)=\int_X a(x, y)b(x, y)f(y)\, d\mu(y),
\ee
 which can be defined, in general, only on some order ideal of $L$. Similarly, the \textit{Hadamard or (Schur) power} $A^{(\alpha)}$ of $A$ is the kernel operator with kernel equal to $(a(x, y))^{\alpha}$ at point $(x, y) \in X \times X$, 
 which can be defined only on some order ideal of $L$.

Let $A_1 , \ldots, A_m$ be positive kernel operators on a Banach function space $L$, and $\alpha_1 , \ldots, \alpha_m$ positive numbers such that $\sum_ {j=1}^m \alpha_j=1$. Then the {\it Hadamard weighted geometric mean} $A=A_1^{(\alpha_1)}\circ A_2^{(\alpha_2)}\circ\cdots\circ A_m^{(\alpha_m)}$ of the operators $A_1 , \ldots, A_m$ is a positive kernel operator defined on the whole space $L$, since $A \le \alpha_1 A_1+ \alpha_2 A_2+ \ldots + \alpha_m A_m$ by the inequality between weighted arithmetic and geometric means.
A matrix $A=[a_{ij}]_{i,j\in R}$ is called {\it non-negative} if $a_{ij}\ge 0$ for all $i, j \in R$. 
For matrices $A$ and $B$, we write $A \ge B$ if the matrix $A - B$ is non-negative.

By an {\it operator} on a Banach sequence space $L$ we always
mean a linear operator on $L$. We say that a non-negative matrix $A$ defines an operator on $L$ if $Ax \in L$ for all $x\in L$, where
$(Ax)_i = \sum _{j \in \NN}a_{ij}x_j$. Then $Ax \in L_+$ for all $x\in L_+$ and
so $A$ defines a {\it positive} operator on $L$. 
Clearly, (finite or infinite) non-negative matrices, that define operators on a Banach sequence spaces are a special case of positive kernel operators.

The following result was proved in \cite[Theorem 2.2]{DP05} and \cite[Theorem 5.1 and Example 3.7]{P06} (see also e.g. \cite[Theorem 2.1]{P17}).

\begin{theorem}
\label{thanfang}
Let $\{A_{i j}\}_{i=1, j=1}^{k, m}$ be positive kernel operators on a Banach function space $L$ and let $\alpha_1$, $\alpha_2$,...,$\alpha_m$ be positive numbers.

\noindent (i) If \; $\sum_{j=1}^{m} \alpha_j=1$, then the positive kernel operator 
\be
A:=\left(A_{1 1}^{(\alpha_1)}\circ\cdots\circ A_{1 m}^{(\alpha_m)} \right)\cdots\left(A_{k 1}^{(\alpha_1)}\circ\cdots\circ A_{k m}^{(\alpha_m)} \right)
\label{ground}
\ee
satisfies the following inequalities
\be
A \le (A_{1 1}\cdots A_{k 1})^{(\alpha_1)}\circ\cdots\circ(A_{1 m}\cdots A_{k m})^{(\alpha_m)} , \\
\label{grund2}
\ee
\begin{align}
\nonumber
\left\|A \right\| \le \left\|(A_{1 1}\cdots A_{k 1})^{(\alpha_1)}\circ\cdots\circ(A_{1 m}\cdots A_{k m})^{(\alpha_m)} \right\| \\
\le \left\|A_{1 1}\cdots A_{k 1}\right\|^{\alpha_1}\cdots \left\|A_{1 m}\cdots A_{k m}\right\|^{\alpha_m}\;\;\;\;\;\;\;\;\;\;\;\;
\label{grund3}
\end{align}
\begin{align}
\nonumber
\rho\left(A\right)\le \rho\left((A_{1 1}\cdots A_{k 1})^{(\alpha_1)}\circ\cdots\circ(A_{1 m}\cdots A_{k m})^{(\alpha_m)}\right) \\
\le \rho\left(A_{1 1}\cdots A_{k 1}\right)^{\alpha_1}\cdots \rho\left(A_{1 m}\cdots A_{k m}\right)^{\alpha_m}\;\;\;\;\;\;\;\;\;\;
\label{grund4}
\end{align}
If, in addition $L$ and $L^*$ have order continuous norms, then
\begin{align}
\nonumber
\gamma\left(A\right)\le \gamma\left((A_{1 1}\cdots A_{k 1})^{(\alpha_1)}\circ\cdots\circ(A_{1 m}\cdots A_{k m})^{(\alpha_m)}\right) \\
\nonumber
\le \gamma\left(A_{1 1}\cdots A_{k 1}\right)^{\alpha_1}\cdots \gamma\left(A_{1 m}\cdots A_{k m}\right)^{\alpha_m}\;\;\;\;\;\;\;\;\;\;
\\
\nonumber
\rho_{ess}\left(A\right)\le \rho_{ess}\left((A_{1 1}\cdots A_{k 1})^{(\alpha_1)}\circ\cdots\circ(A_{1 m}\cdots A_{k m})^{(\alpha_m)}\right) \\
\nonumber
\le \rho_{ess}\left(A_{1 1}\cdots A_{k 1}\right)^{\alpha_1}\cdots \rho_{ess}\left(A_{1 m}\cdots A_{k m}\right)^{\alpha_m}\;\;\;\;\;\;
\label{grund6}
\end{align}
\noindent(ii) If $L\in \mathcal L$, $\sum_{j=1}^{m} \alpha_j\ge 1$ and $\{A_{i j}\}_{i=1, j=1}^{k, m}$ are nonnegative matrices that define operators on $L$, then $A$ from \eqref{ground} defines a positive operator on $L$ and the inequalities \eqref{grund2}, \eqref{grund3} and \eqref{grund4} hold.
\end{theorem}
The following result is a special case of Theorem \ref{thanfang}.
\begin{theorem}
Let $A_1, \ldots, A_m$ be positive kernel operators on a Banach function space $L$, and $\alpha_1, \ldots , \alpha_m$ positive numbers.

\noindent (i) If \; $\sum_{j=1}^{m} \alpha_j=1$ then
\be
\|A_1^{(\alpha_1)}\circ A_2^{(\alpha_2)}\circ\cdots\circ A_m^{(\alpha_m)}\| \le
\|A_1\|^{\alpha_1}\|A_2\|^{\alpha_2}\cdots \|A_m\|^{\alpha_m}
\label{ungl1}
\ee
and 
\be
\rho(A_1^{(\alpha_1)}\circ A_2^{(\alpha_2)}\circ\cdots\circ A_m^{(\alpha_m)}) \le
\rho(A_1)^{\alpha_1} \rho(A_2)^{\alpha_2}\cdots \rho(A_m)^{\alpha_m}
\label{ungl2}
\ee
If, in addition $L$ and $L^*$ have order continuous norms, then
\be
\nonumber
\gamma(A_1^{(\alpha_1)}\circ A_2^{(\alpha_2)}\circ\cdots\circ A_m^{(\alpha_m)}) \le
\gamma(A_1)^{\alpha_1} \gamma(A_2)^{\alpha_2}\cdots \gamma(A_m)^{\alpha_m}
\label{ungl3}
\ee
and
\be
\nonumber
\rho_{ess}(A_1^{(\alpha_1)}\circ A_2^{(\alpha_2)}\circ\cdots\circ A_m^{(\alpha_m)}) \le
\rho_{ess}(A_1)^{\alpha_1} \rho_{ess}(A_2)^{\alpha_2}\cdots \rho_{ess}(A_m)^{\alpha_m}
\label{ungl4}
\ee

\noindent (ii) If $L\in\mathcal L$, $\sum_{j=1}^{m} \alpha_j\ge 1$ and if $A_1, \ldots , A_m$ are nonnegative matrices that define positive operators on $L$, then $A_1^{(\alpha_1)}\circ\cdots\circ A_m^{(\alpha_m)}$ defines positive operator on $L$ and \eqref{ungl1} and \eqref{ungl2} hold.

\noindent (iii) If $L\in\mathcal L$, $t\ge 1$ and if $A, A_1, \ldots , A_m$ are nonnegative matrices that define operators on $L$, then $A^{(t)}$ defines an operator on $L$ and the following inequalities hold
\be
\nonumber
A_1^{(t)}\cdots A_m^{(t)}\le (A_1\cdots A_m)^{(t)}
\label{ungl5}
\ee
\be
\nonumber
\rho(A_1^{(t)}\cdots A_m^{(t)})\le \rho(A_1\cdots A_m)^{t}
\label{ungl6}
\ee
\be
\nonumber
\|A_1^{(t)}\cdots A_m^{(t)}\|\le \|A_1\cdots A_m\|^{t}
\label{ungl7}
\ee
\end{theorem}

\bigskip

Let $\Sigma$ denote a bounded set of bounded operators on a complex Banach space $L$. For $m \ge 1$, $\Sigma^m$ is the set of all products of matrices in $\Sigma$ of length m, i. e.
$$\Sigma^m=\{A_1A_2\cdots A_m : A_i \in \Sigma, i=1, \ldots, m\}.$$
The generalized spectral radius of $\Sigma$ is defined by
\be
\nonumber
\rho(\Sigma)=\limsup_{m \to \infty} \;[\sup_{A\in \Sigma^m} \rho(A)]^{1/m}
\label{rhogen}
\ee
which is equal to 
$$\rho(\Sigma)=\sup_{m \in \NN} \;[\sup_{A\in \Sigma^m} \rho(A)]^{1/m}.$$
The joint spectral radius of $\Sigma$ is defined by 
\be
\nonumber
\hat{\rho} (\Sigma)=\lim_{m \to \infty}[\sup_{A\in \Sigma^m} \|A\|]^{1/m}.
\label{joint}
\ee
The generalized essential spectral radius of $\Sigma$ is defined by
\be
\nonumber
\rho_{ess}(\Sigma)=\limsup_{m \to \infty} \;[\sup_{A\in \Sigma^m} \rho_{ess}(A)]^{1/m}
\label{essgen}
\ee
and is equal to
$$\rho_{ess} (\Sigma)=\sup_{m \in \NN} \;[\sup_{A\in \Sigma^m} \rho_{ess}(A)]^{1/m}.$$
The joint essential spectral radius of $\Sigma$ is defined by
\be
\nonumber
\hat{\rho}_{ess} (\Sigma)=\lim_{m \to \infty}[\sup_{A\in \Sigma^m} \gamma(A)]^{1/m}.
\label{jointess}
\ee
It is well known that $\rho(\Sigma)=\hat{\rho} (\Sigma)$ for a precompact nonempty set $\Sigma$ of compact operators on $L$ (see e.g. \cite{ShT00}, \cite{ShT08}, \cite{M12}), in particular for a bounded set of complex $n \times n$ matrices (see e.g. \cite{BW92}, \cite{MP13}, \cite{E95}, \cite{D11}). This equality is called the Berger-Wang formula or also the generalized spectral radius theorem (for an elegant proof in the finite dimensional case see \cite{D11}). It is known that also the generalized Berger-Wang formula holds, i.e. , that for any precompact nonempty set $\Sigma$ of bounded operators on $L$ we have
$$\hat{\rho} (\Sigma)=\max\{\rho(\Sigma), \hat{\rho}_{ess} (\Sigma)\}$$
(see e.g. \cite{ShT00}, \cite{ShT08}, \cite{M12}).

We will use the following well known facts that hold for all $r\in \{\rho, \hat{\rho}, \rho_{ess}, \hat{\rho}_{ess}\}$:
\be
\nonumber
r(\Sigma^m)=r(\Sigma)^m \;\;\mathrm{and}\;\;
r(\Psi\Sigma)=r(\Sigma\Psi)
\label{erneut}
\ee
where $\Psi\Sigma=\{AB: A\in \Psi, B\in \Sigma\}$ and $m\in \NN$.

Let $\Psi_1, \ldots , \Psi_m$ be bounded sets of positive kernel operators on a Banach function space $L$ and let $\alpha_1, \ldots , \alpha_m$ be positive numbers such that $\sum_{i=1}^{m} \alpha_i=1$. The bounded set of positive kernel operators on $L$, defined by 
$$\Psi_1^{(\alpha_1)}\circ\cdots\circ\Psi_m^{(\alpha_m)}=\{A_1^{(\alpha_1)}\circ\cdots\circ A_m^{(\alpha_m)}: A_1\in \Psi_1, \ldots , A_m\in \Psi_m\},$$
is called the {\it weighted Hadamard (Schur) geometric mean} of sets $\Psi_1, \ldots , \Psi_m$. The set $\Psi_1^{(\frac{1}{m})}\circ\cdots\circ\Psi_m^{(\frac{1}{m})}$ is called the {\it Hadamard (Schur) geometric mean} of sets $\Psi_1, \ldots , \Psi_m$.
The following result, that we will need in the sequel, was proved in \cite[Theorem 3.3]{BP22}.
\begin{theorem}
\label{endlich}
Let $\{\Psi_{i j}\}_{i=1, j=1}^{k, m}$ be bounded sets of positive kernel operators on a Banach function space $L$ and let $\alpha_1, \ldots , \alpha_m$ be positive numbers. \\

\noindent (i) If $r\in\{\rho, \hat{\rho}\}$, $\sum_{i=1}^{m} \alpha_i=1$ and $n\in \NN$, then
\begin{align}
\nonumber
r\left(\left(\Psi_{1 1}^{(\alpha_1)}\circ\cdots\circ\Psi_{1 m}^{(\alpha_m)}\right)\cdots\left(\Psi_{k 1}^{(\alpha_1)}\circ\cdots\circ\Psi_{k m}^{(\alpha_m)}\right)\right)\;\;\;\;\;\;\;\;\;\;\;\;\;\;\;\;\;\;\;\\
\nonumber
\le r\left(\left(\Psi_{1 1}\cdots\Psi_{k 1}\right)^{(\alpha_1)}\circ\cdots\circ\left(\Psi_{1 m}\cdots\Psi_{k m}\right)^{(\alpha_m)}\right)\;\;\;\;\;\;\;\;\;\;\\
\nonumber
\le r\left(((\Psi_{1 1}\cdots\Psi_{k 1})^n)^{(\alpha_1)}\circ\cdots\circ((\Psi_{1 m}\cdots\Psi_{k m})^n)^{(\alpha_m)}\right)^{\frac{1}{n}}\\
\le r(\Psi_{1 1}\cdots\Psi_{k 1})^{\alpha_1}\cdots r(\Psi_{1 m}\cdots\Psi_{k m})^{\alpha_m}\;\;\;\;\;\;\;\;\;\;\;\;\;\;\;\;\;\;\;\;\;\;\;
\label{nice}
\end{align}
If, in addition $L$ and $L^*$ have order continuous norms, the Inequalities \eqref{nice} hold also for each $r\in\{\rho_{ess}, \hat{\rho} _{ess}\}$.

\noindent (ii) If $L\in\mathcal L$, $r\in\{\rho, \hat{\rho}\}$, $\sum_{j=1}^{m} \alpha_j\ge 1$ and $\{\Psi_{i j}\}_{i=1, j=1}^{k, m}$ are bounded sets of nonnegative matrices that define operators on $L$, then Inequalities \eqref{nice} hold.

In particular, if $\Psi_1, \ldots , \Psi_k$ are bounded sets of nonnegative matrices that define positive operators on $L$ and $t\ge 1$, then
\be
\nonumber
r(\Psi_1^{(t)}\cdots\Psi_k^{(t)})\le r((\Psi_1\cdots\Psi_k)^{(t)})\le r(((\Psi_1\cdots\Psi_k)^n)^{(t)})^{\frac{1}{n}}\le r(\Psi_1\cdots\Psi_k)^{t}
\label{mit_t}
\ee
\end{theorem}

The rest of the article is organized in the following way. In Section 2 we extend some results from \cite{P19} by proving generalizations of some inequalities in \cite{P19} to the inequalities for bounded sets of positive kernel operators on Banach function spaces. In Section 3 we give some basic definitions and results concerning the norm of bounded sets of positive kernel operators on $L^2(X, \mu)$, in particular. We prove some results that generalize the results from \cite{P19} and \cite{BP21}.
\section{Results for positive kernel operators on Banach function spaces}
In this section we generalize results from \cite{P19}.

Let $\alpha_1, \alpha_2, \ldots ,\alpha_m$ be positive numbers such that $\sum_{j=1}^{m}\alpha_j=1$ and let $\Psi_1, \ldots ,\Psi_m$ be bounded sets of positive kernel operators on a Banach function space $L$. We define bounded sets of positive kernel operators $\Sigma_1 \ldots , \Sigma_m$ on $L$ in the following way:
\begin{align}\textstyle
\nonumber
\Sigma_1=\Psi_1^{(\alpha_1)}\circ\Psi_2^{(\alpha_2)}\circ\cdots\circ\Psi_m^{(\alpha_m)}\\
\nonumber
\Sigma_2=\Psi_2^{(\alpha_1)}\circ\Psi_3^{(\alpha_2)}\circ\cdots\circ\Psi_1^{(\alpha_m)}\\
\nonumber
\cdots\;\;\;\;\;\;\;\;\;\;\;\;\;\;\;\;\;\;\;\;\;\;\;\;\;\;\;\;\;\;\;\;\;\;\;\;\;\;\;\;\;\;\\
\nonumber
\Sigma_m=\Psi_m^{(\alpha_1)}\circ\Psi_1^{(\alpha_2)}\circ\cdots\circ\Psi_{m-1}^{(\alpha_m)}
\end{align}
i. e. , 
\begin{align}
\Sigma_i=\Psi_i^{(\alpha_1)}\circ\Psi_{i+1}^{(\alpha_2)}\circ\cdots\circ\Psi_m^{(\alpha_{m-i+1})}\circ\Psi_1^{(m-i+2)}\circ\cdots\circ\Psi_{i-1}^{(\alpha_m)}
\label{new2}
\end{align}
for $i=1, \ldots , m$.

The following result that follows from Theorem \ref{endlich}(i) and above definition is an extension of \cite[Theorem 3.1]{P19}. 
\begin{corollary}
\label{kathyth}
Let $\Psi_1, \ldots ,\Psi_m$ be bounded sets of positive kernel operators on a Banach function space $L$ and let $\alpha_1, \ldots ,\alpha_m$ be positive numbers such that $\sum_{j=1}^{m}{\alpha_j}=1$ and let $r\in\{\rho, \hat{\rho}\}$. If $\Sigma_1, \ldots , \Sigma_m$ are bounded sets of positive kernel operators on $L$ defined as in (\ref{new2}) then
\begin{align}\textstyle
r(\Sigma_1  \Sigma_2 \cdots \Sigma_m)\le r(\Psi_1 \Psi_2 \cdots \Psi_m)
\label{new3}
\end{align}
If, in addition $L$ and $L^*$ have order continuous norms, then (\ref{new3}) holds also for each $r\in\{\rho_{ess}, \hat{\rho}_{ess}\}$.
\end{corollary}
\begin{proof}
Applying 
Theorem \ref{endlich}(i) we have
\begin{equation}
\nonumber
r(\Sigma_1 \cdots \Sigma_m)\le r(\Psi_1\Psi_2\cdots\Psi_m)^{\alpha_1}r(\Psi_2\cdots\Psi_m\Psi_1)^{\alpha_2}\cdots r(\Psi_m\Psi_1\cdots\Psi_{m-1})^{\alpha_m}\\
\end{equation}
\begin{equation}
\nonumber
=r(\Psi_1\Psi_2\cdots\Psi_m)\;\;\;\;\;\;\;\;\;\;\;\;\;\;\;\;\;\;\;\;\;\;\;\;\;\;\;\;\;\;\;\;\;\;\;\;\;\;\;\;\;\;\;\;\;\;\;\;
\end{equation}
since $\sum_{j=1}^{m}{\alpha_j}=1$ which completes the proof.
\end{proof}
In fact, applying 
Theorem \ref{endlich}(i) in the proof above we can obtain the following refinement of (\ref{new3}).
\begin{corollary}
\label{Akb}
Let $L$, $r$,$\Psi_j, \Sigma_j$ and $\alpha_j$, for $j=1, \ldots , m$ be as in Theorem \ref{kathyth} and let $\Phi_j=\Psi_j\cdots\Psi_m\Psi_1\cdots\Psi_{j-1}$ for $j=1, \ldots , m$. Then 
\be
\nonumber
r(\Sigma_1  \cdots \Sigma_m)\le r(\Phi_1^{(\alpha_1)}
\circ\cdots\circ\Phi_m^{(\alpha_m)})\le r((\Phi_1^n)^{(\alpha_1)}
\circ\cdots\circ(\Phi_m^n)^{(\alpha_m)})^{1/n}\\
\le r(\Psi_1\cdots\Psi_m)
\label{jedn}
\ee
\end{corollary}
Using 
Theorem \ref{endlich}(ii) instead of 
Theorem \ref{endlich}(i) we obtain the following result which is an extension of the \cite[Theorem 3.6]{P19}.
\begin{corollary}
\label{Bkb}
Let $L\in \mathcal L$, and let $\Psi_1, \ldots , \Psi_m$ be bounded sets of nonnegative matrices that define operators on $L$ and let $\alpha_1, \ldots ,\alpha_m$ be positive numbers such that $\alpha:=\sum_{j=1}^{m}{\alpha_j}\ge1$. If  $\Sigma_1, \ldots ,\Sigma_m$ are families defined by (\ref{new2}) and $\Phi_j=\Psi_j\cdots\Psi_m\Psi_1\cdots\Psi_{j-1}$ for $j=1, \ldots , m$ then for all $r\in\{\rho, \hat{\rho}\}$ the following inequalities hold
\begin{eqnarray}
\nonumber
r(\Sigma_1  \cdots \Sigma_m)\le r(\Phi_1^{(\alpha_1)}
\circ\cdots\circ\Phi_m^{(\alpha_m)})\le r((\Phi_1^n)^{(\alpha_1)}
\circ\cdots\circ(\Phi_m^n)^{(\alpha_m)})^{1/n}\\
\le r(\Psi_1\cdots\Psi_m)^{\alpha}
\label{nejedn}
\end{eqnarray}
\end{corollary}
The following example illustrates that \eqref{nejedn} does not hold in case $\alpha<1$.
\begin{example}{\rm Setting $\Psi_1=\Psi_2=\ldots=\Psi_m=\{A_1\}$, where $A_1=\left[
\begin{matrix}
1 & 1\\
1 & 1\\
\end{matrix}\right]$ and if $r\in\{\rho, \hat{\rho}\}$ and $\Sigma_i=A_1^{(t)}\circ\cdots\circ A_1^{(t)}$ for $i=1, \ldots , m$ then $r(\Sigma_1\cdots\Sigma_m)=2^m>r(\Psi_1\cdots\Psi_m)=(2^m)^{tm}$ for $t<\frac{1}{m}$.}
\end{example}
\section{Results on $L^2(X, \mu)$}
In this section we extend some results from \cite{P19} and \cite{BP21} concerning $L^2(X, \mu)$ spaces.

The notion of joint spectral radius $\hat{\rho}(\Sigma)$ of a bounded sets of operators was introduced by G.-C. Rota and W.G. Strang, 1960, in \cite{RS}, and the norm of the set as the supremum of the norms of its elements (see also \cite{ShT00} and \cite{ShT08}).

Let $\Psi$ be a bounded set of positive kernel operators on $L^2(X, \mu).$ Denote by $\Psi^*$ bounded set of positive kernel operators on $L^2(X, \mu)$ defined by $\Psi^*=\{A^*: A \in\Psi\}.$
\begin{lemma}
Let $\Psi$ be bounded set of positive kernel operators on $L^2(X,\mu).$
 Then we have
\be
\|\Psi\|=\rho(\Psi^*\Psi)^{1/2}=\hat{\rho}(\Psi^*\Psi)^{1/2}
\label{lema}
\ee
\end{lemma}
\begin{proof}
Let $\Psi$ be a bounded set of positive kernel operators on $L^2(X, \mu)$. Using well-known equalities for positive kernel operator on $L^2(X, \mu)$
\be
\nonumber
\|A\|^2=\|AA^*\|=\|A^*A\|=\rho(AA^*)=\rho(A^*A)
\label{poznato}
\ee
we obtain
\begin{eqnarray}
\nonumber
\|\Psi\|=\sup_{A\in\Psi}\|A\|=\sup_{A\in\Psi}\rho(A^*A)^{\frac{1}{2}}=(\sup_{A\in\Psi}\rho((A^*A)^m)^{\frac{1}{m}})^{\frac{1}{2}}\le(\sup_{B\in(\Psi^*\Psi)^m}(\rho(B))^{\frac{1}{m}})^{\frac{1}{2}}=\\
\nonumber
\rho(\Psi^*\Psi)^{\frac{1}{2}}\le\hat{\rho}(\Psi^*\Psi)^{\frac{1}{2}}\le\|\Psi^*\Psi\|^{\frac{1}{2}}\le(\|\Psi^*\|\|\Psi\|)^{\frac{1}{2}}=\|\Psi\|.\;\;\;\;\;\;\;\;\;\;\;\;\;\;\;\;\;\;\;\;\;\;\;\;\;\;\;\;\;\;\;\;\;\;\;\;\;\;
\end{eqnarray} 

\end{proof}
The following theorem generalizes \cite[Theorem 4.6]{P19}. In fact, the same technique of the proof is applied by using Theorem \ref{endlich} instead of Theorem \ref{thanfang}.
\begin{theorem}
\label{first}
Let $\Psi_1, \ldots ,\Psi_m$ be bounded sets of positive kernel operators on $L^2(X, \mu)$ and let $m\in \NN$. For $r \in\{\rho,\hat{\rho}\}$ we have

If $m$ is even, then
\begin{align}
\nonumber
\|\Psi_1^{(\frac{1}{m})}\circ\cdots\circ\Psi_m^{(\frac{1}{m})}\| \le (r(\Psi_1^*\Psi_2\Psi_3^*\Psi_4\cdots\Psi_{m-1}^*\Psi_m)r(\Psi_1\Psi_2^*\Psi_3\Psi_4^*\cdots\Psi_{m-1}\Psi_m^*))^{\frac{1}{2m}}\\
=(r(\Psi_1^*\Psi_2\Psi_3^*\Psi_4\cdots\Psi_{m-1}^*\Psi_m)r(\Psi_m\Psi_{m-1}^*\cdots\Psi_4\Psi_3^*\Psi_2\Psi_1^*))^{\frac{1}{2m}}\;\;\;\;\;\;\;\;\;\;\;\;\;\;\;\;\;\;\;\;\;\;\;\;\;\;\;\;\;\;\;
\label{popravka}
\end{align}
If $m$ is odd, then
\begin{align}
\nonumber
\|\Psi_1^{(\frac{1}{m})}\circ\cdots\circ\Psi_m^{(\frac{1}{m})}\|\;\;\;\;\;\;\;\;\;\;\;\;\;\;\;\;\;\;\;\;\;\;\;\;\;\;\;\;\;\;\;\;\;\;\;\;\;\;\;\;\;\;\;\;\;\;\;\;\;\;\;\;\;\;\;\;\;\;\;\;\;\;\;\;\;\;\;\;\;\;\;\;\;\;\;\;\;\;\;\;\;\;\;\;\;\;\;\;\;\;\;\; \\
\le r^{\frac{1}{2m}}(\Psi_1\Psi_2^*\Psi_3\Psi_4^*\cdots\Psi_{m-2}\Psi_{m-1}^*\Psi_m\Psi_1^*\Psi_2\Psi_3^*\Psi_4\cdots\Psi_{m-2}^*\Psi_{m-1}\Psi_m^*)\;\;\;\;\;\;\;\;\;\;\;\;\;\;\;\;\;\;\;\;\;
\label{ispravka}
\end{align}
\end{theorem}
\begin{proof}
 Let $r\in\{\rho, \hat{\rho}\}$. If $m$ is even, then we obtain by using 
 Theorem \ref{endlich}(i)
$$\left(\left(\Psi_1^{(\frac{1}{m})}\circ\Psi_2^{(\frac{1}{m})}\circ\cdots\circ\Psi_m^{(\frac{1}{m})}\right)^*\left(\Psi_1^{(\frac{1}{m})}\circ\Psi_2^{(\frac{1}{m})}\circ\cdots\circ\Psi_m^{(\frac{1}{m})}\right)\right)^{\frac{m}{2}}=$$
$$\left((\Psi_1^*)^{(\frac{1}{m})}\circ(\Psi_2^*)^{(\frac{1}{m})}\circ\cdots\circ(\Psi_m^*)^{(\frac{1}{m})}\right)\left(\Psi_2^{(\frac{1}{m})}\circ\Psi_3^{(\frac{1}{m})}\circ\cdots\circ\Psi_1^{(\frac{1}{m})}\right)$$
$$\left((\Psi_3^*)^{(\frac{1}{m})}\circ(\Psi_4^*)^{(\frac{1}{m})}\circ\cdots\circ(\Psi_2^*)^{(\frac{1}{m})}\right)\left(\Psi_4^{(\frac{1}{m})}\circ\Psi_5^{(\frac{1}{m})}\circ\cdots\circ\Psi_3^{(\frac{1}{m})}\right)\cdots$$
$$\left((\Psi_{m-1}^*)^{(\frac{1}{m})}\circ(\Psi_m^*)^{(\frac{1}{m})}\circ\cdots\circ(\Psi_{m-2}^*)^{(\frac{1}{m})}\right)\left(\Psi_m^{(\frac{1}{m})}\circ\Psi_1^{(\frac{1}{m})}\circ\cdots\circ\Psi_{m-1}^{(\frac{1}{m})}\right).$$
It follows from (\ref{lema})
\begin{align}
\nonumber
\|\Psi_1^{(\frac{1}{m})}\circ\Psi_2^{(\frac{1}{m})}\circ\cdots\circ\Psi_m^{(\frac{1}{m})}\|^{m}\;\;\;\;\;\;\;\;\;\;\;\;\;\;\;\;\;\;\;\;\;\;\;\;\;\;\;\;\;\;\;\;\;\;\;\;\;\;\;\;\;\;\;\;\;\;\;\;\;\;\;\;\;\;\;\;\;\;\;\;\;\;\;\\
\nonumber
=r\left(\left(\Psi_1^{(\frac{1}{m})}\circ\Psi_2^{(\frac{1}{m})}\circ\cdots\circ\Psi_m^{(\frac{1}{m})}\right)^*\left(\Psi_1^{(\frac{1}{m})}\circ\Psi_2^{(\frac{1}{m})}\circ\cdots\circ\Psi_m^{(\frac{1}{m})}\right)\right)^{\frac{m}{2}}\;\;\;\;\\
\nonumber
\le r(\Sigma)\le r(\Psi_1^*\Psi_2\Psi_3^*\Psi_4\cdots\Psi_{m-1}^*\Psi_m)^{\frac{1}{m}}r(\Psi_2^*\Psi_3\Psi_4^*\Psi_5\cdots\Psi_m^*\Psi_1)^{\frac{1}{m}}\cdots\\
\nonumber
r(\Psi_{m-1}^*\Psi_m\Psi_1^*\Psi_2\cdots\Psi_{m-3}^*\Psi_{m-2})^{\frac{1}{m}}r(\Psi_m^*\Psi_1\Psi_2^*\Psi_3\cdots\Psi_{m-2}^*\Psi_{m-1})^{\frac{1}{m}}\;\;\;\\
\nonumber
=r^{\frac{1}{2}}(\Psi_1^*\Psi_2\Psi_3^*\Psi_4\cdots\Psi_{m-1}^*\Psi_m)r^{\frac{1}{2}}(\Psi_1\Psi_2^*\Psi_3\Psi_4^*\cdots\Psi_{m-1}\Psi_m^*)\;\;\;\;\;\;\;\;\;\;\;\;\;\;\\
\nonumber
(r(\Psi_1^*\Psi_2\Psi_3^*\Psi_4\cdots\Psi_{m-1}^*\Psi_m)r(\Psi_m\Psi_{m-1}^*\cdots\Psi_4\Psi_3^*\Psi_2\Psi_1^*))^{\frac{1}{2}},\;\;\;\;\;\;\;\;\;\;\;\;\;\;\;\;\;
\end{align}
where 
$$\Sigma:=(\Psi_1^*\Psi_2\Psi_3^*\Psi_4\cdots\Psi_{m-1}^*\Psi_m)^{(\frac{1}{m})}\circ(\Psi_2^*\Psi_3\Psi_4^*\Psi_5\cdots\Psi_m^*\Psi_1)^{(\frac{1}{m})}\circ\cdots\circ$$
$$(\Psi_{m-1}^*\Psi_m\Psi_1^*\Psi_2\cdots\Psi_{m-3}^*\Psi_{m-2})^{(\frac{1}{m})}\circ(\Psi_m^*\Psi_1\Psi_2^*\Psi_3\cdots\Psi_{m-2}^*\Psi_{m-1})^{(\frac{1}{m})}$$
which completes the proof of (\ref{popravka}).

If $m$ is odd, we have
\begin{align}
\nonumber
\left(\left(\Psi_1^{(\frac{1}{m})}\circ\Psi_2^{(\frac{1}{m})}\circ\cdots\circ\Psi_m^{(\frac{1}{m})}\right)^*\left(\Psi_1^{(\frac{1}{m})}\circ\Psi_2^{(\frac{1}{m})}\circ\cdots\circ\Psi_m^{(\frac{1}{m})}\right)\right)^{m}\;\;\;\;\;\;\;\;\;\;\;\;\;\;\;\;\;\;\;\;\;\;\;\;\;\;\;\\
\nonumber
=\left((\Psi_1^*)^{(\frac{1}{m})}\circ(\Psi_2^*)^{(\frac{1}{m})}\circ\cdots\circ(\Psi_m^*)^{(\frac{1}{m})}\right)\left(\Psi_2^{(\frac{1}{m})}
\circ\cdots\circ\Psi_m^{(\frac{1}{m})}\circ\Psi_1^{(\frac{1}{m})}\right)\;\;\;\;\;\;\;\;\;\;\;\;\;\;\;\;\;\;\\
\nonumber
\left((\Psi_3^*)^{(\frac{1}{m})}\circ(\Psi_4^*)^{(\frac{1}{m})}\circ\cdots\circ(\Psi_2^*)^{(\frac{1}{m})}\right)\left(\Psi_4^{(\frac{1}{m})}\circ\Psi_5^{(\frac{1}{m})}\circ\cdots\circ\Psi_3^{(\frac{1}{m})}\right)\cdots\;\;\;\;\;\;\;\;\;\;\;\;\;\;\;\;\;\;\\
\nonumber
\left((\Psi_{m-2}^*)^{(\frac{1}{m})}\circ(\Psi_{m-1}^*)^{(\frac{1}{m})}\circ\cdots\circ(\Psi_{m-3}^*)^{(\frac{1}{m})}\right)\left(\Psi_{m-1}^{(\frac{1}{m})}\circ\Psi_m^{(\frac{1}{m})}\circ\cdots\circ\Psi_{m-2}^{(\frac{1}{m})}\right)\;\;\;\;\;\;\;\;\;\\
\nonumber
\left((\Psi_m^*)^{(\frac{1}{m})}\circ(\Psi_1^*)^{(\frac{1}{m})}\circ\cdots\circ(\Psi_{m-1}^*)^{(\frac{1}{m})}\right)\left(\Psi_1^{(\frac{1}{m})}\circ\Psi_2^{(\frac{1}{m})}\circ\cdots\circ\Psi_{m-1}^{(\frac{1}{m})}\circ\Psi_m^{(\frac{1}{m})}\right)\;\;\;\;\;\;\\
\nonumber
\left((\Psi_2^*)^{(\frac{1}{m})}\circ(\Psi_3^*)^{(\frac{1}{m})}\circ\cdots\circ(\Psi_1^*)^{(\frac{1}{m})}\right)\left(\Psi_3^{(\frac{1}{m})}\circ\Psi_4^{(\frac{1}{m})}\circ\cdots\circ\Psi_1^{(\frac{1}{m})}\circ\Psi_2^{(\frac{1}{m})}\right)\cdots\;\;\;\;\;\;\;\\
\nonumber
\left((\Psi_{m-1}^*)^{(\frac{1}{m})}\circ(\Psi_{m}^*)^{(\frac{1}{m})}\circ\cdots\circ(\Psi_{m-2}^*)^{(\frac{1}{m})}\right)\left(\Psi_m^{(\frac{1}{m})}\circ\Psi_1^{(\frac{1}{m})}\circ\cdots\circ\Psi_{m-2}^{(\frac{1}{m})}\circ\Psi_{m-1}^{(\frac{1}{m})}\right)\\
\end{align}
It follows from (\ref{lema}) and 
Theorem \ref{endlich}(i) that
\begin{align}
\nonumber
\|\Psi_1^{(\frac{1}{m})}\circ\Psi_2^{(\frac{1}{m})}\circ\cdots\circ\Psi_m^{(\frac{1}{m})}\|^{2m}=\;\;\;\;\;\;\;\;\;\;\;\;\;\;\;\;\;\;\;\;\;\;\;\;\;\;\;\;\;\;\;\;\;\;\;\;\;\;\;\;\;\;\;\;\;\;\;\;\;\;\;\;\;\;\;\;\;\;\;\;\\
\nonumber
r\left(\left(\Psi_1^{(\frac{1}{m})}\circ\Psi_2^{(\frac{1}{m})}\circ\cdots\circ\Psi_m^{(\frac{1}{m})}\right)^*\left(\Psi_1^{(\frac{1}{m})}\circ\Psi_2^{(\frac{1}{m})}\circ\cdots\circ\Psi_m^{(\frac{1}{m})}\right)\right)^m\le r(\Omega)\;\;\\
\nonumber
r(\Psi_1^*\Psi_2\Psi_3^*\Psi_4\cdots\Psi_{m-1}\Psi_m^*\Psi_1\Psi_2^*\Psi_3\Psi_4^*\cdots\Psi_{m-1}^*\Psi_m)=\;\;\;\;\;\;\;\;\;\;\;\;\;\;\;\;\;\;\;\;\;\;\;\;\;\;\;\;\\
\nonumber
r(\Psi_1\Psi_2^*\Psi_3\cdots\Psi_{m-1}^*\Psi_m\Psi_1^*\Psi_2\cdots\Psi_m^*)\;\;\;\;\;\;\;\;\;\;\;\;\;\;\;\;\;\;\;\;\;\;\;\;\;\;\;\;\;\;\;\;\;\;\;\;\;\;\;\;\;\;\;\;\;\;\;\;\;\;\;\;\;
\end{align}
where
$$\Omega:=(\Psi_1^*\Psi_2\Psi_3^*\Psi_4\cdots\Psi_{m-1}\Psi_m^*\Psi_1\Psi_2^*\Psi_3\cdots\Psi_{m-1}^*\Psi_m)^{(\frac{1}{m})}\circ$$
$$(\Psi_2^*\Psi_3\Psi_4^*\Psi_5\cdots\Psi_m\Psi_1^*\Psi_2\Psi_3^*\cdots\Psi_m^*\Psi_1)^{(\frac{1}{m})}\circ\cdots\circ$$
$$(\Psi_m^*\Psi_1\Psi_2^*\Psi_3\cdots\Psi_{m-1}^*\Psi_m\Psi_1^*\Psi_2\cdots\Psi_{m-2}^*\Psi_{m-1})^{(\frac{1}{m})}$$
which proves (\ref{ispravka}).
\end{proof}
The following theorem generalizes \cite[Theorem 4.8]{P19} and can be proved in a similar way as Theorem \ref{first} by applying 
Theorem \ref{endlich}(ii) instead of 
Theorem \ref{endlich}(i).
\begin{theorem}
\label{thkate}
Let $L\in \mathcal{L}$ and let $\Psi_1, \ldots , \Psi_m$ be bounded sets of nonnegative matrices that define operators on $l^2(R)$ and let $\alpha \ge \frac{1}{m}$ and $r \in \{\rho,\hat{\rho}\}$.

If $m$ is even, then
\begin{align}
\nonumber
\|\Psi_1^{(\alpha)}\circ\Psi_2^{(\alpha)}\circ\cdots\circ\Psi_m^{(\alpha)}\|\le r^{\frac{1}{m}}(\Sigma_{\alpha})\le(r(\Psi_1^*\Psi_2\Psi_3^*\Psi_4\cdots\Psi_{m-1}^*\Psi_m)\\
r(\Psi_m\Psi_{m-1}^*\cdots\Psi_4\Psi_3^*\Psi_2\Psi_1^*))^{\frac{\alpha}{2}},
\label{austausch}
\end{align}
where
\begin{align}
\nonumber
\Sigma_{\alpha}=(\Psi_1^*\Psi_2\Psi_3^*\Psi_4\cdots\Psi_{m-1}^*\Psi_m)^{(\alpha)}\circ(\Psi_2^*\Psi_3\Psi_4^*\Psi_5\cdots\Psi_m^*\Psi_1)^{(\alpha)}\circ\cdots\circ\\
\nonumber
(\Psi_{m-1}^*\Psi_m\Psi_1^*\Psi_2\cdots\Psi_{m-3}^*\Psi_{m-2})^{(\alpha)}\circ(\Psi_m^*\Psi_1\Psi_2^*\Psi_3\cdots\Psi_{m-2}^*\Psi_{m-1})^{(\alpha)}
\end{align}

If $m$ is odd then
\begin{align}
\nonumber
\|\Psi_1^{(\alpha)}\circ\Psi_2^{(\alpha)}\circ\cdots\circ\Psi_m^{(\alpha)}\|\le r^{\frac{1}{2m}}(\Omega_{\alpha})\le\;\;\;\;\;\;\;\;\;\;\;\;\;\;\;\;\;\;\;\\
r^{\frac{\alpha}{2}}(\Psi_1\Psi_2^*\Psi_3\Psi_4^*\cdots\Psi_{m-2}\Psi_{m-1}^*\Psi_m\Psi_1^*\Psi_2\cdots\Psi_{m-1}\Psi_m^*)
\label{ausflug}
\end{align}
where
\begin{align}
\nonumber
\Omega_{\alpha}=(\Psi_1^*\Psi_2\Psi_3^*\Psi_4\cdots\Psi_{m-1}\Psi_m^*\Psi_1\Psi_2^*\Psi_3\cdots\Psi_{m-1}^*\Psi_m)^{(\alpha)}\circ\\
\nonumber
(\Psi_2^*\Psi_3\Psi_4^*\Psi_5\cdots\Psi_{m-1}^*\Psi_m\Psi_1^*\Psi_2\Psi_3^*\cdots\Psi_m^*\Psi_1)^{(\alpha)}\circ\cdots\circ\\
\nonumber
(\Psi_m^*\Psi_1\Psi_2^*\Psi_3\Psi_4^*\cdots\Psi_{m-1}^*\Psi_m\Psi_1^*\Psi_2\cdots\Psi_{m-2}^*\Psi_{m-1})^{(\alpha)}.
\end{align}
\end{theorem}
\begin{example}{\rm Let $\Psi_1=\Psi_2=\Psi_3=\{T_0\}$, where $T_0=\left[
		\begin{matrix}
			0 & 0\\
			1 & 1\\
		\end{matrix}\right]$ and $r\in\{\rho, \hat{\rho}\}.$ Then $\Psi_1^{(\alpha)}\circ\Psi_2^{(\alpha)}\circ\Psi_3^{(\alpha)}=T_1$ for all $\alpha>0$ and $\|T_1\|=\sqrt{2}.$ However $r(\Psi_1\Psi_2^*\Psi_3\Psi_1^*\Psi_2\Psi_3^*)=8$ which shows that \eqref{ausflug} is not valid for $\alpha<\frac{1}{3}.$}
\end{example}
If m is odd, the following result gives a refinement of the inequality \eqref{ispravka} that differs from the refinement given in Theorem \ref{first}. In fact, Theorem 3.5 generalizes \cite[Corollary 3.11]{BP21}. By applying Theorem \ref{endlich} the proof goes similarly as in \cite[Corollary 3.11]{BP21}. 
\begin{theorem}
\label{kety}
Let m be odd and let $\Psi_1, \ldots, \Psi_m$ be bounded sets of positive kernel operators on $L^2(X, \mu)$. For $r\in\{\rho, \hat{\rho}\}$ we have
\begin{align}
\nonumber
\|\Psi_1^{(\frac{1}{m})}\circ\cdots\circ\Psi_m^{(\frac{1}{m})}\|\le\;\;\;\;\;\;\;\;\;\;\;\;\;\;\;\;\;\;\;\;\;\;\;\;\;\;\;\;\;\;\;\;\;\;\;\;\;\;\;\;\;\;\;\;\;\;\;\;\;\;\;\;\;\;\;\;\;\;\;\;\;\;\;\;\;\;\;\;\;\;\;\;\;\;\;\;\;\;\;\;\;\\
\nonumber
	r^{\frac{1}{2}}((\Psi_1\Psi_2^*)^{(\frac{1}{m})}\circ\cdots\circ
(\Psi_{m}\Psi_1^*)^{(\frac{1}{m})}\circ(\Psi_2\Psi_3^*)^{(\frac{1}{m})}\circ\cdots\circ(\Psi_{m-1}\Psi_{m}^*)^{(\frac{1}{m})})\le\;\;\;\;\;\;\;\;\;\\
r^{\frac{1}{2m}}(\Omega_1^{(\frac{1}{m})}\circ\cdots\circ\Omega_m^{(\frac{1}{m})})\le r^{\frac{1}{2m}}(\Psi_1\Psi_2^*
\cdots\Psi_{m-2}\Psi_{m-1}^*\Psi_m\Psi_1^*
\cdots\Psi_{m-2}^*\Psi_{m-1}\Psi_m^*)
\label{laufen}
\end{align}
where 
\begin{align}
\nonumber
\Omega_j=\Psi_{2j-1}\Psi_{2j}^*
\cdots\Psi_{m-2}\Psi_{m-1}^*\Psi_m\Psi_1^*\Psi_2\Psi_3^*\cdots\Psi_{m-1}\Psi_{m}^*\Psi_1\Psi_2^*\cdots\Psi_{2j-3}\Psi_{2j-2}^*
\end{align}
for\; $1\le j\le \frac{m-1}{2}$,
\begin{align}
\nonumber
\Omega_{\frac{m+1}{2}}=\Psi_m\Psi_1^*\Psi_2\Psi_3^*\cdots\Psi_{m-1}\Psi_{m}^*\Psi_1\Psi_2^*\Psi_3\Psi_4^*\cdots\Psi_{m-2}\Psi_{m-1}^*,\;\;\;\;\;\;\;\;\;\;\;\;\;\;\;\;\;\;\;\;\;\;\;\;
\end{align}
\begin{align}
\nonumber
\Omega_j=\Psi_{2j-m-1}\Psi_{2j-m}^*
\cdots\Psi_{m-1}\Psi_m^*\Psi_1\Psi_2^*\Psi_3\Psi_4^*\cdots\Psi_m\Psi_1^*\cdots\Psi_{2j-m-3}\Psi_{2j-m-2}^*
\end{align}
for\; $\frac{m+3}{2}\le j \le m$.
\end{theorem}
\begin{proof}
Let $r\in\{\rho, \hat{\rho}\}$. By Lemma \ref{lema} and commutativity of Hadamard product using 
Theorem \ref{endlich}(i) we have
\begin{align}
\nonumber
\|\Psi_1^{(\frac{1}{m})}\circ\cdots\circ\Psi_m^{(\frac{1}{m})}\|^2=r((\Psi_1^{(\frac{1}{m})}\circ\cdots\circ\Psi_m^{(\frac{1}{m})})(\Psi_1^{(\frac{1}{m})}\circ\cdots\circ\Psi_m^{(\frac{1}{m})})^*)
=\;\;\;\;\;\;\;\;\;\;\;\;\;\;\;\;\;\;\;\;\\
\nonumber
r(\!(\Psi_1^{(\frac{1}{m})}\circ\cdots\circ
\Psi_m^{(\frac{1}{m})}\circ\Psi_2^{(\frac{1}{m})}\circ\cdots\circ\Psi_{m-1}^{(\frac{1}{m})})((\Psi_2^*)^{(\frac{1}{m})}\circ\cdots\circ
(\Psi_1^*)^{(\frac{1}{m})}\circ
\cdots\circ(\Psi_{m}^*)^{(\frac{1}{m})})\!)\\
\nonumber
\le r((\Psi_1\Psi_2^*)^{(\frac{1}{m})}
\circ\cdots\circ(\Psi_{m-2}\Psi_{m-1}^*)^{(\frac{1}{m})}\circ(\Psi_m\Psi_1^*)^{(\frac{1}{m})}\circ\cdots\circ(\Psi_{m-1}\Psi_m^*)^{(\frac{1}{m})})\;\;\;\;\;\;\;\;\;\;\;\;
\end{align}
Notice that
\begin{align}
\nonumber
((\Psi_1\Psi_2^*)^{(\frac{1}{m})}
\circ\cdots\circ(\Psi_{m-2}\Psi_{m-1}^*)^{(\frac{1}{m})}\circ(\Psi_m\Psi_1^*)^{(\frac{1}{m})}\circ\cdots\circ(\Psi_{m-1}\Psi_m^*)^{(\frac{1}{m})})^{m}=\;\;\;\;\;\;\;\;\;\;\;\\
\nonumber
((\Psi_1\Psi_2^*)^{(\frac{1}{m})}
\circ\cdots\circ(\Psi_{m-2}\Psi_{m-1}^*)^{(\frac{1}{m})}\circ(\Psi_m\Psi_1^*)^{(\frac{1}{m})}\circ\cdots\circ(\Psi_{m-1}\Psi_m^*)^{(\frac{1}{m})})\times\;\;\;\;\;\;\;\;\;\;\;\;\;\;\\
\nonumber
((\Psi_3\Psi_4^*)^{(\frac{1}{m})}
\circ\cdots\circ(\Psi_{m}\Psi_{1}^*)^{(\frac{1}{m})}\circ(\Psi_2\Psi_3^*)^{(\frac{1}{m})}\circ\cdots\circ(\Psi_{1}\Psi_2^*)^{(\frac{1}{m})})\cdots\;\;\;\;\;\;\;\;\;\;\;\;\;\;\;\;\;\;\;\;\;\;\;\;\;\;\\
\nonumber
((\Psi_{m-1}\Psi_m^*)^{(\frac{1}{m})}\circ(\Psi_1\Psi_2^*)^{(\frac{1}{m})}\circ\cdots\circ(\Psi_2\Psi_3^*)^{(\frac{1}{m})}\circ\cdots\circ(\Psi_{m-3}\Psi_{m-2}^*)^{(\frac{1}{m})})\;\;\;\;\;\;\;\;\;\;\;\;\;\;\;\;\;\;	
\end{align}
It follows that by 
Theorem \ref{endlich}(i)
\begin{align}
\nonumber
r((\Psi_1\Psi_2^*)^{(\frac{1}{m})}
\circ\cdots\circ(\Psi_{m-2}\Psi_{m-1}^*)^{(\frac{1}{m})}\circ(\Psi_m\Psi_1^*)^{(\frac{1}{m})}\circ\cdots\circ(\Psi_{m-1}\Psi_m^*)^{(\frac{1}{m})})\le\;\;\;\;\;\;\;\;\;\;\;\\
\nonumber
r(\Omega_1^{(\frac{1}{m})}\circ\cdots\circ\Omega_m^{(\frac{1}{m})})^{\frac{1}{m}}\le (r(\Omega_1)\cdots r(\Omega_m))^{\frac{1}{m^2}}=\;\;\;\;\;\;\;\;\;\;\;\;\;\;\;\;\;\;\;\;\;\;\;\;\;\;\;\;\;\;\;\;\;\;\;\;\;\;\;\;\;\;\;\;\;\;\;\;\;\;\;\\
\nonumber
r(\Psi_1\Psi_2^*\Psi_3\Psi_4^*
\cdots\Psi_{m-2}\Psi_{m-1}^*\Psi_m\Psi_1^*\Psi_2\Psi_3^*
\cdots\Psi_{m-1}\Psi_m^*)^{\frac{1}{m}}\;\;\;\;\;\;\;\;\;\;\;\;\;\;\;\;\;\;\;\;\;\;\;\;\;\;\;\;\;\;\;\;\;\;\;\;\;\;\;\;
\end{align}
since $r(\Omega_1)=\ldots = r(\Omega_m)$, which completes the proof.
\end{proof}
The proof of the following result goes similarly by applying 
Theorem \ref {endlich}(ii) instead of 
Theorem \ref{endlich}(i). It refines the inequality \eqref{ausflug} in a different way as  in Theorem \ref{thkate}.
\begin{theorem}
\label{kate}
Let $m\in\NN$ be odd and let $\Psi_1, \ldots , \Psi_m$ be bounded sets of nonnegative matrices that define operators on $l^2(R).$ If $\alpha\ge\frac{1}{m}$ and\; $\Omega_1, \ldots , \Omega_m$ are sets defined in Theorem \ref{kety}, then for $r\in\{\rho, \hat{\rho}\}$
\begin{align}
\nonumber
\|\Psi_1^{(\alpha)}\circ\Psi_2^{(\alpha)}
\circ\cdots\circ\Psi_m^{(\alpha)}\|\le\;\;\;\;\;\;\;\;\;\;\;\;\;\;\;\;\;\;\;\;\;\;\;\;\;\;\;\;\;\;\;\;\;\;\;\;\;\;\;\;\;\;\;\;\;\;\;\;\;\;\;\;\;\;\;\;\;\;\;\;\;\;\;\;\;\;\;\;\;\;\;\;\;\;\;\;\;\;\;\;\;\;\\
\nonumber
r((\Psi_1\Psi_2^*)^{(\alpha)}\circ\cdots\circ(\Psi_{m-2}\Psi_{m-1}^*)^{(\alpha)}\circ(\Psi_m\Psi_1^*)^{(\alpha)}
\circ\cdots\circ(\Psi_{m-1}\Psi_m^*)^{(\alpha)})^{\frac{1}{2}}\le\;\;\;\;\;\;\;\;\;\;\;\;\;\;\\
r(\Omega_1^{(\alpha)}\circ\cdots\circ\Omega_m^{(\alpha)})^{\frac{1}{2m}}\le r(\Psi_1\Psi_2^*\Psi_3\Psi_4^*\cdots\Psi_{m-2}\Psi_{m-1}^*\Psi_m\Psi_1^*\Psi_2\Psi_3^*\cdots\Psi_{m-1}\Psi_m^*)^{\frac{\alpha}{2}}\;\;\;\;\;
\label{urlaub}
\end{align}
\end{theorem}
The following corollary extends the result \cite[Corollary 3.12]{BP21}.
\begin{corollary}
\noindent{\rm (i)} Let $\Psi$ and $\Sigma$ be bounded sets of positive kernel operators on $L^2(X, \mu)$, $r\in\{\rho, \hat{\rho}\}$. Then
\begin{align}
\nonumber
\|\Psi^{(\frac{1}{3})}\circ(\Sigma^*)^{(\frac{1}{3})}\circ\Psi^{(\frac{1}{3})}\|
\le r^{\frac{1}{2}}((\Psi^*\Sigma^*)^{(\frac{1}{3})}\circ(\Psi^*\Psi)^{(\frac{1}{3})}\circ(\Sigma\Psi)^{(\frac{1}{3})})\le\;\;\;\;\;\;\;\;\;\;\;\;\;\;\;\;\;\;\;\;\;\;\;\;\;\;\\
r^{\frac{1}{6}}((\Psi^*\Sigma^*\Psi^*\Psi\Sigma\Psi)^{(\frac{1}{3})}\circ(\Psi^*\Psi\Sigma\Psi\Psi^*\Sigma^*)^{(\frac{1}{3})}\circ(\Sigma\Psi\Psi^*\Sigma^*\Psi^*\Psi)^{(\frac{1}{3})})\le\|\Psi\Sigma\Psi\|^{\frac{1}{3}}\;\;\;\;\;\;\;\;\;
\label{henne}
\end{align}
\noindent{\rm (ii)} If $\Psi$ and $\Sigma$ are bounded sets of nonnegative matrices that define operators on $l^2(R)$ and if $\alpha\ge\frac{1}{3}$ then
\begin{align}
\nonumber
\|\Psi^{(\alpha)}\circ(\Sigma^*)^{(\alpha)}\circ\Psi^{(\alpha)}\|\le r^{\frac{1}{2}}((\Psi^*\Sigma^*)^{(\alpha)}\circ(\Psi^*\Psi)^{(\alpha)}\circ(\Sigma\Psi)^{(\alpha)})\le\;\;\;\;\;\;\;\;\;\;\;\;\;\;\;\;\;\;\;\;\;\;\;\;\;\;\;\;\;\;\;\;\;\;\;\;\;\;\;\;\;\;\;\;\;\;\;\;\;\;\;\;\;\;\;\;\;\;\;\\
r^{\frac{1}{6}}((\Psi^*\Sigma^*\Psi^*\Psi\Sigma\Psi)^{(\alpha)}\circ(\Psi^*\Psi\Sigma\Psi\Psi^*\Sigma^*)^{(\alpha)}\circ(\Sigma\Psi\Psi^*\Sigma^*\Psi^*\Psi)^{(\alpha)})\le\|\Psi\Sigma\Psi\|^{\alpha}\;\;\;\;\;\;\;\;\;\;\;\;\;\;\;\;\;\;\;\;\;\;\;\;\;\;\;\;\;\;\;\;\;\;\;\;\;\;\;\;\;\;
\label{huhn}
\end{align}
\end{corollary}
\begin{proof}
The inequalities \eqref{henne} and \eqref{huhn} follow from Theorem \ref{kety} and Theorem \ref{kate} respectively by taking $\Psi_1=\Psi^*$, $\Psi_2=\Sigma$ and $\Psi_3=\Psi^*$, using the fact that $\|\Psi\|=\|\Psi^*\|$, applying 
Theorem \ref{endlich}(i) and 
Theorem \ref{endlich}(ii), respectively, as well as Lemma \ref{lema}. 
\end{proof}
The following result and its proof extend  \cite[Theorem 3.5]{BP21}.
\begin{theorem}
\label{ferkel}
 Let m be even, $\tau, \nu\in S_m$, 
  and let $\Psi_1, \ldots , \Psi_m$ be bounded sets of positive kernel operators on $L^2(X, \mu).$ Denote $\Sigma_j=\Psi_{\tau(2j-1)}^*\Psi_{\tau(2j)}$ and $\Sigma_{\frac{m}{2}+j}=\Psi_{\tau(2j)}^*\Psi_{\tau(2j-1)}=\Sigma_{j}^*$ for $j=1, \ldots , \frac{m}{2}$. Let $\Omega_i=\Sigma_{\nu(i)}\cdots\Sigma_{\nu(m)}\Sigma_{\nu(1)}\cdots\Sigma_{\nu(i-1)}$ for $i=1, \ldots , m$ and $r\in\{\rho, \hat{\rho}\}$
  
\noindent{\rm (i)} Then
\begin{align}
\nonumber
\|\Psi_1^{(\frac{1}{m})}\circ\cdots\circ\Psi_m^{(\frac{1}{m})}\|\le r^{\frac{1}{2}}(\Sigma_1^{(\frac{1}{m})}\circ\cdots\circ\Sigma_m^{(\frac{1}{m})})\;\;\;\;\;\;\;\;\;\;\;\;\;\\
\le r^{\frac{1}{2m}}(\Omega_1^{(\frac{1}{m})}\circ\cdots\circ\Omega_m^{(\frac{1}{m})})\le r^{\frac{1}{2m}}(\Sigma_{\nu(1)}\cdots\Sigma_{\nu(m)})
\label{sau}
\end{align}
\noindent{\rm (ii)} Let $\Psi_1, \ldots , \Psi_m$ be bounded sets of nonnegative matrices that define operators on $l^2(R)$. If $\alpha\ge\frac{1}{m}$ then
\begin{align}
\nonumber
\|\Psi_1^{(\alpha)}\circ\cdots\circ\Psi_m^{(\alpha)}\|\le r^{\frac{1}{2}}(\Sigma_1^{(\alpha)}\circ\cdots\circ\Sigma_m^{(\alpha)})\;\;\;\;\;\;\;\;\;\;\;\;\;\;\\
\le r^{\frac{1}{2m}}(\Omega_1^{(\alpha)}\circ\cdots\circ\Omega_m^{(\alpha)})\le r(\Sigma_{\nu(1)}\cdots\Sigma_{\nu(m)})^{\frac{\alpha}{2}}
\label{schwein}
\end{align}
\end{theorem}
\begin{proof}
By Lemma \ref{lema} and using commutativity of Hadamard product we have
\begin{align}
\nonumber
\|\Psi_1^{(\frac{1}{m})}\circ\cdots\circ\Psi_m^{(\frac{1}{m})}\|=r((\Psi_1^{(\frac{1}{m})}\circ\cdots\circ\Psi_m^{(\frac{1}{m})})^*(\Psi_1^{(\frac{1}{m})}\circ\cdots\circ\Psi_m^{(\frac{1}{m})}))^{\frac{1}{2}}\;\;\;\;\\
\nonumber
=r(((\Psi_{\tau(1)}^*)^{(\frac{1}{m})}\circ\cdots\circ(\Psi_{\tau(m-1)}^*)^{(\frac{1}{m})}\circ(\Psi_{\tau(2)}^*)^{(\frac{1}{m})}\circ\cdots\circ(\Psi_{\tau(m)}^*)^{(\frac{1}{m})})\\
\nonumber
\cdot(\Psi_{\tau(2)}^{(\frac{1}{m})}\circ\cdots\circ\Psi_{\tau(m)}^{(\frac{1}{m})}\circ\Psi_{\tau(1)}^{(\frac{1}{m})}\circ\cdots\circ\Psi_{\tau(m-1)}^{(\frac{1}{m})}))^{\frac{1}{2}}\;\;\;\;\;\;\;\;\;\;\;\;\;\;\;\;\;\;\;\;\;\;\;\;\;\;\;\;\;\;\\
\nonumber
\le r((\Psi_{\tau(1)}^*\Psi_{\tau(2)})^{(\frac{1}{m})}\circ\cdots\circ(\Psi_{\tau(m-1)}^*\Psi_{\tau(m)})^{(\frac{1}{m})}\circ(\Psi_{\tau(2)}^*\Psi_{\tau(1)})^{(\frac{1}{m})}\;\;\;\;\;\;\\
\nonumber
\circ\cdots\circ(\Psi_{\tau(m)}^*\Psi_{\tau(m-1)})^{(\frac{1}{m})})^{\frac{1}{2}}\;\;\;\;\;\;\;\;\;\;\;\;\;\;\;\;\;\;\;\;\;\;\;\;\;\;\;\;\;\;\;\;\;\;\;\;\;\;\;\;\;\;\;\;\;\;\;\;\;\;\;\;\;\;\;\;\\
\nonumber
=r(\Sigma_1^{(\frac{1}{m})}\circ\cdots\circ\Sigma_m^{(\frac{1}{m})})^{\frac{1}{2}}=r(\Sigma_{\nu(1)}^{(\frac{1}{m})}\circ\cdots\circ\Sigma_{\nu(m)}^{(\frac{1}{m})})^{\frac{1}{2}},\;\;\;\;\;\;\;\;\;\;\;\;\;\;\;\;\;\;\;\;\;\;\;\;\;\;\;
\end{align}
which proves the first inequality in \eqref{sau}. To prove the second and the third inequality in \eqref{sau} observe that
\begin{align}
\nonumber
(\Sigma_{\nu(1)}^{(\frac{1}{m})}\circ\cdots\circ\Sigma_{\nu(m)}^{(\frac{1}{m})})^m=(\Sigma_{\nu(1)}^{(\frac{1}{m})}\circ\Sigma_{\nu(2)}^{(\frac{1}{m})}\circ\cdots\circ\Sigma_{\nu(m)}^{(\frac{1}{m})})(\Sigma_{\nu(2)}^{(\frac{1}{m})}\circ(\Sigma_{\nu(3)}^{(\frac{1}{m})}\circ\cdots\circ\Sigma_{\nu(1)}^{(\frac{1}{m})})\\
\nonumber
\cdots(\Sigma_{\nu(m)}^{(\frac{1}{m})}\circ\Sigma_{\nu(1)}^{(\frac{1}{m})}\circ\cdots\circ\Sigma_{\nu(m-1)}^{(\frac{1}{m})})\;\;\;\;\;\;\;\;\;\;\;\;\;\;\;\;\;\;\;\;\;\;\;\;\;\;\;\;\;\;\;
\end{align}
and
\begin{align}
\nonumber
r(\Sigma_{\nu(1)}^{(\frac{1}{m})}\circ\cdots\circ\Sigma_{\nu(m)}^{(\frac{1}{m})})\le r((\Sigma_{\nu(1)}\Sigma_{\nu(2)}\cdots\Sigma_{\nu(m)})^{(\frac{1}{m})}\circ(\Sigma_{\nu(2)}\Sigma_{\nu(3)}\cdots\Sigma_{\nu(1)})^{(\frac{1}{m})}\\
\nonumber
\circ\cdots\circ(\Sigma_{\nu(m)}\Sigma_{\nu(1)}\cdots\Sigma_{\nu(m-1)})^{(\frac{1}{m})})^{\frac{1}{m}}=r(\Omega_1^{(\frac{1}{m})}\circ\cdots\circ\Omega_m^{(\frac{1}{m})})^{\frac{1}{m}}
\end{align}
since $r(\Sigma^m)=r(\Sigma)^m$ and by 
Theorem \ref{endlich}(i). It follows that
\begin{align}
\nonumber
r(\Sigma_{\nu(1)}^{(\frac{1}{m})}\circ\cdots\circ\Sigma_{\nu(m)}^{(\frac{1}{m})})^{\frac{1}{2}}\le r(\Omega_1^{(\frac{1}{m})}\circ\cdots\circ\Omega_m^{(\frac{1}{m})})^{\frac{1}{2m}}\le (r(\Omega_1)\cdots r(\Omega_m))^{\frac{1}{2m^2}}\\
\nonumber
=r(\Sigma_{\nu(1)}\cdots\Sigma_{\nu(m)})^{\frac{1}{2m}}\;\;\;\;\;\;\;\;\;\;\;\;\;\;\;\;\;\;\;\;\;\;\;\;\;\;\;\;\;\;\;\;\;\;\;\;\;\;\;
\end{align}
according to 
Theorem \ref{endlich}(i) and using the fact that $r(\Omega_1)= \ldots =r(\Omega_m)\!=r(\Sigma_{\nu(1)}\cdots\Sigma_{\nu(m)})$ which completes the proof of \eqref{sau}. Inequalities \eqref{schwein} are proved in a similar way by applying 
Theorem \ref{endlich}(ii).
\end{proof}
\begin{remark}\label{kbgd}
In the special case, taking permutation $\nu$ such that $\nu(j)=j$ for $j=1, \ldots , \frac{m}{2}$ and $\nu(\frac{m}{2}+j)=m-j+1$ for $j=1, \ldots , \frac{m}{2}$  in Theorem \ref{ferkel} we obtain $r^{\frac{1}{2}}(\Sigma_{\nu(1)}\cdots\Sigma_{\nu(m)})=r^{\frac{1}{2}}(\Sigma_1\cdots\Sigma_{\frac{m}{2}}\Sigma_m\cdots\Sigma_{\frac{m}{2}+1})=r^{\frac{1}{2}}(\Sigma_{1}\cdots\Sigma_{\frac{m}{2}}\Sigma_{\frac{m}{2}}^{*}\cdots\Sigma_{1}^{*})$
$=\|\Psi_{\tau(1)}^*\Psi_{\tau(2)}\Psi_{\tau(3)}^*\Psi_{\tau(4)}\cdots\Psi_{\tau(m-1)}^*\Psi_{\tau(m)}\|$ by Lemma \ref{lema}.
\end{remark}
\begin{example}{\rm Let $m=4$, $\Psi_1=\Psi_2=\Psi_3=\Psi_4=\{T_0\}$, where $T_0=\left[
		\begin{matrix}
			0 & 0\\
			1 & 1\\
		\end{matrix}\right]$, $\Sigma_i$ families defined in Theorem \ref{ferkel} for $i=1, \ldots , 4$, $\alpha>0$ and let $\tau$ be identity permutation and $\nu$ permutation defined in Remark \ref{kbgd}. Then $\|\Psi_1^{(\alpha)}\circ\Psi_2^{(\alpha)}\circ\Psi_3^{(\alpha)}\circ\Psi_4^{(\alpha)}\|=\sqrt{2}$, and $r^{\frac{\alpha}{2}}(\Sigma_{\nu(1)}\Sigma_{\nu(2)}\Sigma_{\nu(3)}\Sigma_{\nu(4)})=\|\Psi_{1}^*\Psi_{2}\Psi_{3}^*\Psi_{4}\|^{\alpha}=4^{\alpha}$ which shows that \eqref{schwein} does not hold for $\alpha<\frac{1}{4}$.}
\end{example}
The following result and its proof is generalization of \cite[Theorem 3.8]{BP21} and its proof.
\begin{theorem}
Let $m\in\NN$ be even, $\alpha\ge\frac{2}{m}$, $\tau\in S_m$ and let $\Psi_1, \ldots , \Psi_m$ be bounded sets of nonnegative matrices that define operators on $l^2(R)$. Let $\Sigma_j$ for $j=1, \ldots , m$ be as in Theorem \ref{ferkel} and denote $\Theta_i=\Sigma_i\cdots\Sigma_{\frac{m}{2}}\Sigma_1\cdots\Sigma_{i-1}$ for $i=1, \ldots , \frac{m}{2}$. If $r\in\{\rho, \hat{\rho}\}$ then
$$\|\Psi_1^{(\alpha)}\circ\cdots\circ\Psi_m^{(\alpha)}\|\le r(\Sigma_1^{(\alpha)}\circ\cdots\circ\Sigma_m^{(\alpha)})^{\frac{1}{2}}\le r(\Sigma_1^{(\alpha)}\circ\cdots\circ\Sigma_{\frac{m}{2}}^{(\alpha)})$$
$$=r((\Psi_{\tau(1)}^*\Psi_{\tau(2)})^{(\alpha)}\circ(\Psi_{\tau(3)}^*\Psi_{\tau(4)})^{(\alpha)}\circ\cdots\circ(\Psi_{\tau(m-1)}^*\Psi_{\tau(m)})^{(\alpha)})$$
\be
\le r(\Theta_1^{(\alpha)}\circ\Theta_2^{(\alpha)}\circ\cdots\circ\Theta_{\frac{m}{2}}^{(\alpha)})^{\frac{2}{m}}\le
r(\Psi_{\tau(1)}^*\Psi_{\tau(2)}\Psi_{\tau(3)}^*\Psi_{\tau(4)}\cdots\Psi_{\tau(m-1)}^*\Psi_{\tau(m)})^{\alpha}
\label{taube}
\ee
\end{theorem}
\begin{proof}
By the first inequality in \eqref{schwein} and using 
Theorem \ref{endlich}(ii) we have
$$\|\Psi_1^{(\alpha)}\circ\cdots\circ\Psi_m^{(\alpha)}\|\le r(\Sigma_1^{(\alpha)}\circ\cdots\circ\Sigma_m^{(\alpha)})^{\frac{1}{2}}$$
$$=r(\Sigma_1^{(\alpha)}\circ\cdots\circ\Sigma_{\frac{m}{2}}^{(\alpha)}\circ(\Sigma_1^*)^{(\alpha)}\circ\cdots\circ(\Sigma_{\frac{m}{2}}^*)^{(\alpha)})^{\frac{1}{2}}$$
$$\le (r(\Sigma_1^{(\alpha)}\circ\cdots\circ\Sigma_{\frac{m}{2}}^{(\alpha)})r((\Sigma_1^{(\alpha)}\circ\cdots\circ\Sigma_{\frac{m}{2}}^{(\alpha)})^*))^{\frac{1}{2}}=r(\Sigma_1^{(\alpha)}\circ\cdots\circ\Sigma_{\frac{m}{2}}^{(\alpha)})$$
$$=r((\Psi_{\tau(1)}^*\Psi_{\tau(2)})^{(\alpha)}\circ(\Psi_{\tau(3)}^*\Psi_{\tau(4)})^{(\alpha)}\circ\cdots\circ(\Psi_{\tau(m-1)}^*\Psi_{\tau(m)})^{(\alpha)})$$
Since
$$((\Psi_{\tau(1)}^*\Psi_{\tau(2)})^{(\alpha)}\circ(\Psi_{\tau(3)}^*\Psi_{\tau(4)})^{(\alpha)}\circ\cdots\circ(\Psi_{\tau(m-1)}^*\Psi_{\tau(m)})^{(\alpha)})^{\frac{m}{2}}=$$
$$((\Psi_{\tau(1)}^*\Psi_{\tau(2)})^{(\alpha)}
\circ\cdots\circ(\Psi_{\tau(m-1)}^*\Psi_{\tau(m)})^{(\alpha)})((\Psi_{\tau(3)}^*\Psi_{\tau(4)})^{(\alpha)}
\circ\cdots\circ(\Psi_{\tau(1)}^*\Psi_{\tau(2)})^{(\alpha)})$$
$$\cdots((\Psi_{\tau(m-1)}^*\Psi_{\tau(m)})^{(\alpha)}
\circ\cdots\circ(\Psi_{\tau(m-3)}^*\Psi_{\tau(m-2)})^{(\alpha)})$$
It follows by 
Theorem \ref{endlich}(ii)
$$r((\Psi_{\tau(1)}^*\Psi_{\tau(2)})^{(\alpha)}\circ(\Psi_{\tau(3)}^*\Psi_{\tau(4)})^{(\alpha)}\circ\cdots\circ(\Psi_{\tau(m-1)}^*\Psi_{\tau(m)})^{(\alpha)})\le$$
$$r(\Theta_1^{(\alpha)}\circ\Theta_2^{(\alpha)}\circ\cdots\circ\Theta_{\frac{m}{2}}^{(\alpha)})^{\frac{2}{m}}\le (r(\Theta_1)^{\alpha}\cdots r(\Theta_{\frac{m}{2}})^{\alpha})^{\frac{2}{m}}$$
$$=r(\Psi_{\tau(1)}^*\Psi_{\tau(2)}\Psi_{\tau(3)}^*\Psi_{\tau(4)}\cdots\Psi_{\tau(m-1)}^*\Psi_{\tau(m)})^{\alpha},$$
where the last equality follows from $r(\Theta_1)=\cdots =r(\Theta_{\frac{m}{2}}).$
\end{proof}
\begin{example}{\rm Let $m=4$, $\Psi_1=\Psi_4=\biggl\{\left[
		\begin{matrix}
			0 & 0\\
			1 & 1\\
		\end{matrix}\right]\biggr\}$, $\Psi_2=\biggl\{\left[
	    \begin{matrix}
	    	1 & 0\\
	    	1 & 1\\
	    \end{matrix}\right]\biggr\}$, $\Psi_3=\biggl\{\left[
        \begin{matrix}
        	0 & 1\\
        	1 & 1\\
        \end{matrix}\right]\biggr\}$, $\alpha>0$ and let $\tau$ be identity permutation. Then $\|\Psi_1^{(\alpha)}\circ\Psi_2^{(\alpha)}\circ\Psi_3^{(\alpha)}\circ\Psi_4^{(\alpha)}\|=\sqrt{2}$, while $r(\Psi_{\tau(1)}^*\Psi_{\tau(2)}\Psi_{\tau(3)}^*\Psi_{\tau(4)})^{\alpha}=3^{\alpha}$ which shows that \eqref{taube} is not valid for $\alpha<\frac{1}{2}\log_{3} 2$.}
\end{example}
The following result is generalization of \cite[Theorem 3.10]{BP21}.
\begin{theorem}
\label{stern}
	Let $\Psi_1, \ldots , \Psi_m$ be bounded sets of positive kernel operators on $L^2(X, \mu)$ and $\tau, \nu\in S_m$. Denote  $\Omega_j=\Psi_{\tau(j)}^*\Psi_{\nu(j)}\cdots\Psi_{\tau(m)}^*\Psi_{\nu(m)}
	\cdots\Psi_{\tau(j-1)}^*\Psi_{\nu(j-1)}$ for $j=1, \ldots , m$. Let $r\in\{\rho, \hat{\rho}\}$.

\noindent{\rm (i)} Then
$$\|\Psi_1^{(\frac{1}{m})}\circ\cdots\circ\Psi_m^{(\frac{1}{m})}\|\le r((\Psi_{\tau(1)}^*\Psi_{\nu(1)})^{(\frac{1}{m})}\circ\cdots\circ(\Psi_{\tau(m)}^*\Psi_{\nu(m)})^{(\frac{1}{m})})^{\frac{1}{2}}$$
\be
\le r((\Omega_1)^{(\frac{1}{m})}\circ\cdots\circ(\Omega_m)^{(\frac{1}{m})})^{\frac{1}{2m}}\le r(\Psi_{\tau(1)}^*\Psi_{\nu(1)}\cdots\Psi_{\tau(m)}^*\Psi_{\nu(m)})^{\frac{1}{2m}}
\label{esel}
\ee
\noindent{\rm (ii)} If $\Psi_1, \ldots , \Psi_m$ are bounded sets of nonnegative matrices that define operators on $l^2(R)$ and if $\alpha\ge\frac{1}{m}$, then
$$\|\Psi_1^{(\alpha)}\circ\cdots\circ\Psi_m^{(\alpha)}\|\le r((\Psi_{\tau(1)}^*\Psi_{\nu(1)})^{(\alpha)}\circ\cdots\circ(\Psi_{\tau(m)}^*\Psi_{\nu(m)})^{(\alpha)})^{\frac{1}{2}}$$
\be
\le r(\Omega_1^{(\alpha)}\circ\cdots\circ\Omega_m^{(\alpha)})^{\frac{1}{2m}}\le r(\Psi_{\tau(1)}^*\Psi_{\nu(1)}\cdots\Psi_{\tau(m)}^*\Psi_{\nu(m)})^{\frac{\alpha}{2}}
\label{affe}
\ee 
\end{theorem}
\begin{proof}
We prove \eqref{esel}. By Lemma \ref{lema}, by commutativity of Hadamard product and by 
Theorem \ref{endlich}(i) we have
$$\|\Psi_1^{(\frac{1}{m})}\circ\cdots\circ\Psi_m^{(\frac{1}{m})}\|=r((\Psi_1^{(\frac{1}{m})}\circ\cdots\circ\Psi_m^{(\frac{1}{m})})^*(\Psi_1^{(\frac{1}{m})}\circ\cdots\circ\Psi_m^{(\frac{1}{m})}))^{\frac{1}{2}}=$$
$$r(((\Psi_{\tau(1)}^*)^{(\frac{1}{m})}\circ\cdots\circ(\Psi_{\tau(m)}^*)^{(\frac{1}{m})})(\Psi_{\nu(1)}^{(\frac{1}{m})}\circ\cdots\circ\Psi_{\nu(m)}^{(\frac{1}{m})}))^{\frac{1}{2}}$$
$$\le r((\Psi_{\tau(1)}^*\Psi_{\nu(1)})^{(\frac{1}{m})}\circ\cdots\circ(\Psi_{\tau(m)}^*\Psi_{\nu(m)})^{(\frac{1}{m})})^{\frac{1}{2}}$$
Observe that
$$((\Psi_{\tau(1)}^*\Psi_{\nu(1)})^{(\frac{1}{m})}\circ\cdots\circ(\Psi_{\tau(m)}^*\Psi_{\nu(m)})^{(\frac{1}{m})})^m=((\Psi_{\tau(1)}^*\Psi_{\nu(1)})^{(\frac{1}{m})}\circ$$
$$\cdots\circ(\Psi_{\tau(m)}^*\Psi_{\nu(m)})^{(\frac{1}{m})})\cdot$$
$$((\Psi_{\tau(2)}^*\Psi_{\nu(2)})^{(\frac{1}{m})}\circ\cdots\circ(\Psi_{\tau(1)}^*\Psi_{\nu(1)})^{(\frac{1}{m})})\cdots((\Psi_{\tau(m)}^*\Psi_{\nu(m)})^{(\frac{1}{m})}\circ$$
$$\cdots\circ(\Psi_{\tau(m-1)}^*\Psi_{\nu(m-1)})^{(\frac{1}{m})}).$$
It follows by 
Theorem \ref{endlich}(i) that
$$r((\Psi_{\tau(1)}^*\Psi_{\nu(1)})^{(\frac{1}{m})}\circ\cdots\circ(\Psi_{\tau(m)}^*\Psi_{\nu(m)})^{(\frac{1}{m})})^{\frac{1}{2}}\le r(\Omega_1^{(\frac{1}{m})}\circ\cdots\circ\Omega_m^{(\frac{1}{m})})^{\frac{1}{2m}}$$
$$\le (r(\Omega_1)\cdots r(\Omega_m))^{\frac{1}{2m^2}}=r(\Psi_{\tau(1)}^*\Psi_{\nu(1)}\cdots\Psi_{\tau(m)}^*\Psi_{\nu(m)})^{\frac{1}{2m}}.$$
The last equality follows from $r(\Omega_1)= \ldots = r(\Omega_m)$ which completes the proof of \eqref{esel}. The proof of \eqref{affe} follows similarly by applying 
Theorem \ref{endlich}(ii). 
\end{proof}
\begin{example}{\rm Let $m=4$, $\Psi_1=\Psi_4=\biggl\{\left[
		\begin{matrix}
			0 & 0\\
			1 & 1\\
		\end{matrix}\right]\biggr\}$, $\Psi_2=\biggl\{\left[
		\begin{matrix}
			1 & 0\\
			1 & 1\\
		\end{matrix}\right]\biggr\}$, $\Psi_3=\biggl\{\left[
		\begin{matrix}
			0 & 1\\
			1 & 1\\
		\end{matrix}\right]\biggr\}$, $\alpha>0$, $\tau=\bigl(\begin{smallmatrix}
	                                           1 & 2 & 3 & 4\\
                                               4 & 3 & 2 & 1
                                               \end{smallmatrix}\bigr)$, $\nu=\bigl(\begin{smallmatrix}
                                               1 & 2 & 3 & 4\\
                                               2 & 1 & 4 & 3
                                           \end{smallmatrix}\bigr)$. It follows that $\|\Psi_1^{(\alpha)}\circ\Psi_2^{(\alpha)}\circ\Psi_3^{(\alpha)}\circ\Psi_4^{(\alpha)}\|=\sqrt{2}$, while $r(\Psi_{\tau(1)}^*\Psi_{\nu(1)}\Psi_{\tau(2)}^*\Psi_{\nu(2)}\Psi_{\tau(3)}^*\Psi_{\nu(3)}\Psi_{\tau(4)}^*\Psi_{\nu(4)})^{\frac{\alpha}{2}}=4^{\alpha}$, which shows that \eqref{affe} does not hold for values of $\alpha<\frac{1}{4}$. }
\end{example}
The following result is consequence of Theorem \ref{stern} and gives different refinement of the inequalities \eqref{laufen} and \eqref{urlaub} as in Theorem \ref{kety} and Theorem \ref{kate} and generalizes \cite[Corollary 5.11]{BP21}.
\begin{corollary}
Let m be odd and let $\Psi_1, \ldots , \Psi_m$ be bounded sets of positive kernel operators on $L^2(X, \mu)$ and let $\Omega_j$ for $j=1, \ldots, m$ be as in Theorem \ref{stern} and let $r\in\{\rho, \hat{\rho}\}$.

\noindent{\rm (i)} Then
$$\|\Psi_1^{(\frac{1}{m})}\circ\cdots\circ\Psi_m^{(\frac{1}{m})}\|$$
$$\le r((\Psi_1^*\Psi_2)^{(\frac{1}{m})}\circ\cdots\circ(\Psi_{m-2}^*\Psi_{m-1})^{(\frac{1}{m})}\circ(\Psi_m^*\Psi_1)^{(\frac{1}{m})}\circ(\Psi_2^*\Psi_3)^{(\frac{1}{m})}\circ\cdots\circ$$
$$(\Psi_{m-1}^*\Psi_{m})^{(\frac{1}{m})})^{\frac{1}{2}}\le r(\Omega_1^{(\frac{1}{m})}\circ\cdots\circ\Omega_m^{(\frac{1}{m})})^{\frac{1}{2m}}$$
$$\le r(\Psi_1^*\Psi_2\cdots\Psi_{m-2}^*\Psi_{m-1}\Psi_m^*\Psi_1\Psi_2^*\Psi_3\cdots\Psi_{m-1}^*\Psi_{m})^{\frac{1}{2m}}$$
\be
\nonumber
=r(\Psi_1\Psi_2^*\Psi_3\cdots\Psi_{m-1}^*\Psi_{m}\Psi_1^*\Psi_2\cdots\Psi_{m-2}^*\Psi_{m-1}\Psi_m^*)^{\frac{1}{2m}}
\label{sonne}
\ee

\noindent{\rm (ii)} If $\Psi_1, \ldots , \Psi_m$ are nonnegative matrices that define operators on $l^2(R)$ and if $\alpha\ge\frac{1}{m}$ then
$$\|\Psi_1^{(\alpha)}\circ\cdots\circ\Psi_m^{(\alpha)}\|$$
$$\le r((\Psi_1^*\Psi_2)^{(\alpha)}\circ\cdots\circ(\Psi_{m-2}^*\Psi_{m-1})^{(\alpha)}\circ(\Psi_m^*\Psi_1)^{(\alpha)}\circ(\Psi_2^*\Psi_3)^{(\alpha)}\circ$$
$$\cdots\circ(\Psi_{m-1}^*\Psi_{m})^{(\alpha)})^{\frac{1}{2}} \le r(\Omega_1^{(\alpha)}\circ\cdots\circ\Omega_m^{(\alpha)})^{\frac{1}{2m}}$$
$$\le r(\Psi_1^*\Psi_2\cdots\Psi_{m-2}^*\Psi_{m-1}\Psi_m^*\Psi_1\Psi_2^*\Psi_3\cdots\Psi_{m-1}^*\Psi_{m})^{\frac{\alpha}{2}}$$
\be
\nonumber
=r(\Psi_1\Psi_2^*\Psi_3\cdots\Psi_{m-1}^*\Psi_{m}\Psi_1^*\Psi_2\cdots\Psi_{m-2}^*\Psi_{m-1}\Psi_m^*)^{\frac{\alpha}{2}}
\label{wind}
\ee 
\end{corollary}
\begin{proof}
It follows by taking the permutations $\tau(j)=2j-1$ for $1\le j \le\frac{m+1}{2}$; $\tau(j)=2(j-\frac{m+1}{2})$ for $\frac{m+3}{2}\le j \le m$ and $\nu(j)=2j$ for $1\le j \le\frac{m-1}{2}$; $\nu(j)=2(j-\frac{m-1}{2})-1$ for $\frac{m+1}{2}\le j \le m$ in Theorem \ref{stern}.
\end{proof}
We conclude this article with the generalization of \cite[Lemma 3.13]{BP21}. 

\begin{lemma}
\label{dog}
Let $\alpha\ge\frac{1}{2}$ and let $\Psi$ be bounded set of nonnegative matrices that define operators on $l^2(R)$ and $r\in\{\rho, \hat{\rho}\}$. Then
\be
r(\Psi^{(\alpha)}\circ(\Psi^*)^{(\alpha)})\le r(\Psi^{(\alpha)}\circ\Psi^{(\alpha)})\le r(\Psi)^{2\alpha}
\label{hund}
\ee
\end{lemma}
\begin{proof}
To prove the first inequality in \eqref{hund} observe that
 $\Psi^{(\alpha)}\circ(\Psi^*)^{(\alpha)}\subset (\Psi^{(\alpha)}\circ\Psi^{(\alpha)})^{(\frac{1}{2})}\circ((\Psi^*)^{(\alpha)}\circ(\Psi^*)^{(\alpha)})^{(\frac{1}{2})}$ since $A^{(\alpha)}\circ(B^*)^{(\alpha)}=(A^{(\alpha)}\circ A^{(\alpha)})^{(\frac{1}{2})}\circ ((B^*)^{(\alpha)}\circ(B^*)^{(\alpha)})^{(\frac{1}{2})}$ for all $A, B\in\Psi$. It follows that $r(\Psi^{(\alpha)}\circ(\Psi^*)^{(\alpha)})\le r((\Psi^{(\alpha)}\circ\Psi^{(\alpha)})^{(\frac{1}{2})}\circ((\Psi^*)^{(\alpha)}\circ(\Psi^*)^{(\alpha)})^{(\frac{1}{2})})\le r(\Psi^{(\alpha)}\circ\Psi^{(\alpha)})^{\frac{1}{2}}r((\Psi^*)^{(\alpha)}\circ(\Psi^*)^{(\alpha)})^{\frac{1}{2}}=r(\Psi^{(\alpha)}\circ\Psi^{(\alpha)})\le r(\Psi)^{2\alpha}$ by \cite[Theorem 3.3(ii)]{BP21}. 
\end{proof}
The following result generalizes \cite[Corollary 3.14]{BP21}. It follows from Theorem \ref{stern} and Lemma \ref{dog}.
\begin{corollary}
Let $\alpha\ge\frac{1}{2}$ and let $\Psi$ and $\Sigma$ be bounded sets of nonnegative matrices that define operators on $l^2(R)$. If $r\in\{\rho, \hat{\rho}\}$ then
$$\|\Psi^{(\alpha)}\circ\Sigma^{(\alpha)}\|\le r((\Psi^*\Sigma)^{(\alpha)}\circ(\Sigma^*\Psi)^{(\alpha)})^{\frac{1}{2}}$$
\be
\nonumber
\le r((\Psi^*\Sigma)^{(\alpha)}\circ(\Psi^*\Sigma)^{(\alpha)})^{\frac{1}{2}}\le r(\Psi^*\Sigma)^{\alpha}
\label{matrix}
\ee
\end{corollary}

\noindent {\bf Acknowledgements.}
I would like to express my thanks to professor Aljo\v{s}a Peperko for suggesting the content and form of this section, as well as notations used there and throughout the paper. I also thank him for his support during the preparation of this paper and encouragement during the work.  The author thanks the colleagues and the staff at the Faculty of Mechanical Engineering and Institute of Mathematics, Physics and Mechanical Engineering in Ljubljana for their hospitality during the research stay in Slovenia.

This work was partial supported by grant KA103 of Erasmus+ European Mobility program, action 18232 of COST Short Term Scientific Mission program and by grants P1-0222 ana P1-0288 of the Slovenian Research Agency.

\bibliographystyle{amsplain}

\end{document}